\documentclass[12pt]{amsart}
\usepackage{amssymb}
\usepackage[shortlabels]{enumitem}
\usepackage[all]{xy}
\usepackage{xcolor}
\usepackage{mathrsfs}
\usepackage{stmaryrd}
\usepackage[margin=1in]{geometry} 
\usepackage[bookmarks, bookmarksdepth=2, colorlinks=true, linkcolor=blue, citecolor=blue, urlcolor=blue]{hyperref}
\usepackage{eucal}
\usepackage{contour}
\usepackage[normalem]{ulem}
\usepackage{array}
\usepackage{multirow}
\usepackage{booktabs}
\usepackage{afterpage}

\newcolumntype{A}{ >{$} r <{$} @{} >{${}} l <{$} }

\makeatletter
\@addtoreset{equation}{section}
\makeatother

\numberwithin{equation}{section}
\newtheorem{theorem}[equation]{Theorem}

\newtheorem{proposition}[equation]{Proposition}

\newtheorem{corollary}[equation]{Corollary}

\theoremstyle{definition}
\newtheorem{rmk}[equation]{Remark}
\newenvironment{remark}[1][]{\begin{rmk}[#1] \pushQED{\qed}}{\popQED \end{rmk}}

\newcommand{\bF}{\mathbf{F}}

\newcommand{\fF}{\mathfrak{F}}

\newcommand{\bN}{\mathbf{N}}

\newcommand{\fP}{\mathfrak{P}}

\newcommand{\bQ}{\mathbf{Q}}

\newcommand{\bZ}{\mathbf{Z}}

\newcommand{\fb}{\mathfrak{b}}

\newcommand{\arxiv}[1]{\href{http://arxiv.org/abs/#1}{{\tiny\tt arXiv:#1}}}
\newcommand{\DOI}[1]{\href{http://doi.org/#1}{\color{purple}{\tiny\tt DOI:#1}}}

\contourlength{1pt}

\contourlength{0.8pt}
\newcommand{\myuline}[1]{%
  \uline{\phantom{#1}}%
  \llap{\contour{white}{#1}}%
}
\DeclareMathOperator{\uRep}{\text{\myuline{\rm Rep}}}
\DeclareMathOperator{\uPerm}{\underline{Perm}}

\renewcommand{\phi}{\varphi}
\DeclareMathOperator{\Aut}{Aut}

\renewcommand{\Vec}{\mathrm{Vec}}
\newcommand{\defn}[1]{\emph{#1}}
\let\lbb\llbracket
\let\rbb\rrbracket

\newcommand\m[1]{$\dot{#1}$}
\def\mm#1{\r{#1}}

\title{The thirty-seven measures on permutations}

\author{Andrew Snowden}
\thanks{The author was supported by NSF grant DMS-2301871.}
\date{April 12, 2024}

\begin{document}

\begin{abstract}
In recent work with Harman, we introduced a notion of measure on a class of finite relational structures. In this note, we consider measures on the class of permutations, i.e., finite sets with two total orders. Using a method of Nekrasov, we show that there are exactly~37 measures. Two of these measures lead to a new pre-Tannakian tensor category.
\end{abstract}

\maketitle
\tableofcontents

\section{Introduction}

\subsection{Measures} \label{ss:meas}

Let $\fF$ be a class of finite relational structures. For example, $\fF$ could be the class of finite totally ordered sets, or finite graphs; see \cite{Macpherson} for general background. Given embeddings $i \colon Y \to X$ and $j \colon Y \to Y'$ of structures in $\fF$, an \defn{amalgamation} is a structure $X'$ in $\fF$ equipped with embeddings $i' \colon Y' \to X'$ and $j' \colon X \to X'$ that agree on $Y$ and are jointly surjective. Amalgamations figure prominently in the theory of Fra\"iss\'e limits.

A \defn{measure} on the class $\fF$ valued in a commutative ring $k$ is a rule $\mu$ that assigns to each embedding $i \colon Y \to X$ in $\fF$ a value $\mu(i)$ in $k$ such that the following three conditions hold:
\begin{enumerate}
\item We have $\mu(i)=1$ if $i$ is an isomorphism.
\item We have $\mu(j \circ i) = \mu(j) \cdot \mu(i)$ when defined.
\item Suppose $i \colon Y \to X$ and $j \colon Y \to Y'$ are embeddings, and let $(X'_{\alpha}, i'_{\alpha}, j'_{\alpha})$, for $1 \le \alpha \le n$, be the various amalgamations (up to isomorphism)\footnote{We assume that there are always finitely many amalgamations (up to isomorphism).}. Then $\mu(i)=\sum_{\alpha=1}^n \mu(i'_{\alpha})$.
\end{enumerate}
There is a universal measure valued in a ring $\Theta(\fF)$. To define $\Theta(\fF)$, start with the polynomial ring in formal symbols $[i]$, for embeddings $i$ in $\fF$, and quotient by the ideal generated by relations corresponding to the above axioms. Giving a $k$-valued measure on $\fF$ is then equivalent to giving a ring homomorphism $\Theta(\fF) \to k$. Thus computing the ring $\Theta(\fF)$ is an important problem, if one cares about measures.

The notion of measure originated in our joint work with Harman \cite{repst}, though similar notions have also been considered \cite{AMSW,KS,Wolf}. We showed how one can construct a tensor category from a measure, and used this observation to produce some fundamentally new examples of tensor categories. Due to this application, we are very interested in measures. However, constructing and classifying measures seems to be a difficult problem in general. In this note, we solve this problem for a specific class $\fF$, namely, the class of permutations.

\subsection{Permutations}

A \defn{permutation} is a finite set equipped with two total orders, which we denote by $<$ and $\prec$. Permutations can be conveniently denoted by a sequence of distinct numbers; the first order is the usual one, while the second comes from the order in the sequence. For instance, 3142 denotes the permutation with elements $\{1,2,3,4\}$ and orders
\begin{displaymath}
1 < 2 < 3 < 4, \qquad 3 \prec 1 \prec 4 \prec 2.
\end{displaymath}
Every permutation is isomorphic to a unique \defn{normalized} permutation, i.e., one using each of the numbers $1, \ldots, n$, for some $n$. The study of permutations has a long history, going back to classical works such as \cite{ES, Knuth, Schensted, SS}, and continues to be an active field today; see \cite{Vatter} for an overview.

\subsection{The main result}

The main theorem of this paper determines the measures on the class $\fP$ of permutations:

\begin{theorem}
There exist thirty-seven $\bZ$-valued measures on $\fP$ (see Appendix~\ref{s:meas}). Together, they define a ring isomorphism $\Theta(\fP) \cong \bZ^{37}$.
\end{theorem}

We note that all of the measures take values in $\{-1,0,1\}$. To prove the theorem, we apply a method of Nekrasov \cite{Nekrasov} (reviewed in \S \ref{s:bg}) that reduces the problem to a finite computation. We work through this computation, with the aid of a computer in places.

\begin{remark}
Every measure takes the value~0 on some embedding. This means that there are no regular measures on $\fP$, in the terminology of \cite{repst}. The non-existence of regular measures had previously been proved by Harman.
\end{remark}

\begin{remark}
As stated, the proof of the theorem involves some computer calculations. Our final answer is consistent with a number of properties that we can prove theoretically. This makes us reasonably confident that our answer is correct.
\end{remark}

\subsection{Applications to tensor categories}

Pre-Tannakian categories are perhaps the nicest symmetric tensor categories. It is quite difficult to construct new examples, and this is a major problem in the subject; one of the main points of \cite{repst} is that measures sometimes lead to new examples. Each measure $\mu$ on $\fP$ yields a tensor category $\uPerm(\fP; \mu)$, as discussed in \S \ref{s:tensor}, which is not pre-Tannakian. An important problem is to determine if these categories have abelian envelopes (which, by definition, are pre-Tannakian). We do not have any ideas about this at the moment.

While our measures on $\fP$ do not directly yield any new pre-Tannakian categories, two of these measures do yield a new pre-Tannakian category via a process called semi-simplification. To state this result, we prefer to use the language of oligomorphic groups. Let $G=\Aut(\bQ,<)$ be the group of automorphisms of the orderd set $(\bQ,<)$. In \cite{repst}, we showed that $G$ carries a unique regular measure $\mu$, and that the category $\uRep_k(G; \mu)$ defined in \cite[Part~III]{repst} is semi-simple and  pre-Tannakian (for any field $k$). This category is now called the \defn{Delannoy category}. It was studied in depth in \cite{line}, where it was found to have a number of remarkable properties; see \cite{circle} for an analogous case.

We obtain the following result:

\begin{corollary} \label{cor:wreath}
The wreath product $G \wr G$ admits a regular measure $\mu$. For any field $k$, the category $\uRep_k(G \wr G; \mu)$ is a semi-simple pre-Tannakian category.
\end{corollary}

It seems quite likely that there might be a good theory of measures on wreath products of oligomorphic groups in general; however, this has not been developed yet.

\subsection{Open problems}

(a) If $\mu$ is one of our 37 measures then $\mu$ assigns a value in $\{-1,0,1\}$ to each inclusion of permutations $Y \subset X$. Is there a meaningful combinatorial interpretation of these values?

(b) Let $\fF_n$ be the class of finite sets endowed with $n$ total orders. We can show $\Theta(\fF_n) \cong \bZ^{a(n)}$ for some $a(n) \in \bN$. In \cite{repst}, we found $a(1)=4$, and in this paper we find $a(2)=37$. A natural problem is to determine $a(n)$ in general. Nekrasov's method applies in principal to each $\fF_n$, and reduces the problem to a finite computation. However, for $\fF_3$ there are 1,999,581 minimal marked structures (compared to the 87 we have for $\fF_2$), so this may already be out of reach of brute force calculation.

(c) A \defn{permutation class} is a class of permutations that is hereditary (see \S \ref{ss:fraisse}). Such classes have been widely studied. In this paper, we classify measures on the class of all permutations. Are there any interesting measures on other permutation classes?

(d) One can also consider ``circular permutations,'' that is, sets endowed with two cyclic orders. What are the measures for them?

\subsection*{Acknowledgments}

We thank Nate Harman, Sarah Kitchen, Ilia Nekrasov and Steven Sam for helpful discussions.

\section{Nekrasov's presentation} \label{s:bg}

In \S \ref{s:bg}, we review some material about measures on Fra\"iss\'e classes. In particular, we describe the presentation of $\Theta(\fF)$ due to Nekrasov. We refer to \cite{arboreal} for proofs.

\subsection{Fra\"iss\'e classes} \label{ss:fraisse}

Let $\fF$ be a class of finite relational structures. We say that $\fF$ is \defn{hereditary} if it is closed under isomorphism (i.e., $X \in \fF$ and $Y \cong X$ implies $Y \in \fF$) and passing to substructures (i.e., $X \in \fF$ and $Y \subset X$ implies $Y \in \fF$). Note that we allow the empty structure, so if a hereditary class is non-empty then it must contain the empty structure. We say that $\fF$ satisfies the \defn{amalgamation property} if whenever one has embeddings $Y \to X$ and $Y \to Y'$ in $\fF$, there is at least one amalgamation that belongs to $\fF$.

We say that $\fF$ is a \defn{Fra\"iss\'e class} if it satisfies the following conditions:
\begin{enumerate}
\item $\fF$ is non-empty.
\item $\fF$ is hereditary.
\item $\fF$ satisfies the amalgamation property.
\item $\fF$ has only finitely many isomorphism classes of $n$ element structures, for each $n$.
\end{enumerate}
Condition (d) is often relaxed to simply asking that $\fF$ has only countably many isomorphism classes of structures, but the above will be more convenient for us. Note that some sources also require $\fF$ to satisfy the joint embedding property, but this follows from the amalgamation property since we include the empty structure.

We assume in the remainder of \S \ref{s:bg} that $\fF$ is a Fra\"iss\'e class. While the definiton of measure makes sense more broadly, some results we present depend on the Fra\"iss\'e condition.

\subsection{A presentation}

The ring $\Theta(\fF)$ is, by definition, generated by classes $[i]$, where $i \colon Y \to X$ is an embedding of structures in $\fF$. We say that $i$ is a \defn{one-point extension} if $\# X=1+\#Y$. Clearly, every embedding factors into a sequence of one-point extensions, and so by \S \ref{ss:meas}(b) the ring $\Theta(\fF)$ is generated by the classes of such embeddings. It will be convenient to adopt a slightly different perspective. A \defn{marked structure} is a pair $(X,x)$ where $X$ is a member of $\fF$ and $x \in X$. We define $[X,x]$ to be the class in $\Theta(\fF)$ of the one-point extension $X \setminus x \to X$. Thus $\Theta(\fF)$ is generated by the classes of marked structures.

We now describe the relations between marked structures. To this end, let $P$ be the polynomial ring over $\bZ$ with variables corresponding to isomorphism classes of marked structures. For a marked structure $(X,x)$, we let $\lbb X,x \rbb$ denote the corresponding variable of $P$. We let $\phi \colon P \to \Theta(\fF)$ be the ring homomorphism given by $\phi(\lbb X,x \rbb) = [X,x]$.

An \defn{L-datum} is a triple $(X,a,b)$ where $X$ is a member of $\fF$ and $a$ and $b$ are distinct points of $X$. Let $X_1, \ldots, X_n$ be the proper amalgamations of $X \setminus a$ and $X \setminus b$ over $X \setminus \{a,b\}$, where ``proper'' means that $a$ and $b$ are distinct in the amalgamation. Define $\delta=1$ if there is an improper amalgamation, and $\delta=0$ otherwise; note that $\delta=1$ if any only if there is an isomorphism $X \setminus b \to X \setminus a$ that is the identity away from $a$ and $b$. We define
\begin{displaymath}
L(X,a,b) = \lbb X \setminus b,a \rbb - \delta - \sum_{i=1}^n \lbb X_i, a \rbb.
\end{displaymath}
This element belongs to the kernel of $\phi$ by \S \ref{ss:meas}(c). We let $\fb_1$ be the ideal of $P$ generated by $L(X,a,b)$ over all choices of L-data. Note that $a$ and $b$ are not symmetrical in $L$, i.e., $L(X,a,b)$ and $L(X,b,a)$ may not be equal.

A \defn{Q-datum} is a triple $(X,a,b)$, where again $X$ is a member of $\fF$ and $a$ and $b$ are distinct points of $X$. Given such a datum, we define
\begin{displaymath}
Q(X,a,b) = \lbb X, a \rbb \cdot \lbb X \setminus a, b \rbb - \lbb X, b \rbb \cdot \lbb X \setminus b, a \rbb.
\end{displaymath}
This belongs to the kernel of $\phi$ by \S \ref{ss:meas}(b). We let $\fb_2$ be the ideal of $P$ generated by $Q(X,a,b)$ over all choices of Q-data. Note that $a$ and $b$ are symmetric in $Q$, i.e., $Q(X,a,b)=Q(X,b,a)$.

Let $\fb=\fb_1+\fb_2$. The following is an initial presentation of $\Theta$:

\begin{proposition}
The map $\phi$ induces an isomorphism $P/\fb \cong \Theta(\fF)$.
\end{proposition}

\begin{proof}
See \cite[Proposition~2.7]{arboreal}.
\end{proof}

\begin{remark}
The definition of L-datum here differs slightly from the one in \cite[\S 2.4]{arboreal}, but they give rise to the same linear equations and the same $\fb_1$ ideal. We have adopted the present definition as it leads to a concise notation for L-data.
\end{remark}

\subsection{A refined presentation}

Let $X$ be a structure in $\fF$ and let $x,y \in X$ be distinct elements. We say that $x$ and $y$ are \defn{separated} if $X$ is the unique amalgamation in $\fF$ of $X \setminus x$ and $X \setminus y$ over $X \setminus \{x,y\}$. In this case, \S \ref{ss:meas}(c) implies that we have an equality
\begin{displaymath}
[X, x] = [X \setminus y, x]
\end{displaymath}
in $\Theta(\fF)$. This motivates the following definition: we say that a marked structure $(X,x)$ is \defn{minimal} if there is no point in $X$ that is separated from the marked point $x$. The above identity shows that the class of any marked structure is equal to the class of a minimal marked structure, and so $\Theta(\fF)$ is generated by the classes of minimal marked structures. This was an important observation of Nekrasov \cite{Nekrasov}.

We now aim to describe the relations among this new generating set. Let $P^* \subset P$ be the polynomial ring on the variables corresponding to minimal marked structures, and let $\phi^* \colon P^* \to \Theta(\fF)$ be the restriction of $\phi$ to $P^*$. Thus our goal is to give generators for $\ker(\phi^*)$.

A \defn{modification} of a marked structure $(X,x)$ is a marked structure $(X',x)$ obtained by adding or removing a single vertex that is separated from $x$. We say that two marked structures are \defn{equivalent} if one can be obtained from the other by a chain of modifications (or isomorphisms). We let $\fb_0^*$ be the ideal of $P^*$ generated by the elements $\lbb X,x \rbb-\lbb X',x' \rbb$ whenever $(X,x)$ and $(X',x')$ are equivalent minimal marked structures. Clearly, $\fb_0^*$ belongs to the kernel of $\phi$. Note that there is a well-defined map $\pi \colon P \to P^*/\fb_0^*$ that takes $\lbb X,x \rbb$ to $\lbb X',x' \rbb$, where $(X',x')$ is any minimal marked structure equivalent to $(X,x)$.

We say that an L-datum $(X,a,b)$ is \defn{minimal} if $a$ and $b$ are not separated in $X$, and there is no point of $X \setminus b$ that is stably separated from $a$. Here we say that two points $x,y \in X$ are \defn{stably separated} if for any embedding $i \colon X \to X'$ in $\fF$ the points $i(x)$ and $i(y)$ are separated. We define $\fb_1^*$ to be the inverse image in $P^*$ of the ideal of $P^*/\fb_0^*$ generated by the elements $\pi(L(X,a,b))$, where $(X,a,b)$ is a minimal L-datum.

We say that a Q-datum $(X,a,b)$ is \defn{minimal} if the following conditions hold:
\begin{enumerate}
\item It is not symmetric, i.e., there is no automorphism of $X$ switching $a$ and $b$ and fixing all other points.
\item $a$ and $b$ are not separated.
\item There is no $c \in X \setminus \{a,b\}$ such that $c$ is separated from the marked vertex in each of $(X,a)$, $(X,b)$, $(X \setminus b, a)$, and $(X \setminus a, b)$.
\end{enumerate}
We let $\fb_2^*$ be the inverse image in $P^*$ of the ideal of $P^*/\fb_0^*$ generated by the elements $\pi(Q(X,a,b))$, where $(X,a,b)$ is a minimal Q-datum.

Put $\fb^*=\fb_1^*+\fb_2^*$. The following is our refined presentation of $\Theta$:

\begin{proposition}
The map $\phi^*$ induces an isomorphism $P^*/\fb^* \to \Theta(\fF)$.
\end{proposition}

\begin{proof}
See \cite[Theorem~2.16]{arboreal}.
\end{proof}

\section{The main result} \label{s:main}

In \S \ref{s:main} we determine the ring $\Theta(\fP)$ for the class $\fP$ of permutations by applying the method of Nekrasov from \S \ref{s:bg}.

\subsection{Amalgamations} \label{ss:amalg}

Let $(X,<,\prec)$ be a permutation and let $x \ne y$ be elements of $X$. We say that $x$ and $y$ are \defn{$<$-adjacent} if they are adjacent in the $<$ order, meaning there is no $z$ such that $x<z<y$ or $y<z<x$. We similarly define \defn{$\prec$-adjacent}. We say that $x$ and $y$ are \defn{neighbors} if they are $<$-adjacent or $\prec$-adjacent (or both). For example, in the permutation 153264, the neighbors of 3 are 2, 4, and 5; indeed, 2 is $<$- and $\prec$-adjacent to 3, 4 is $<$-adjacent to 3, and 5 is $\prec$-adjacent to~3.

Consider a proper amalgamation $(X',<',\prec')$ of $X \setminus x$ and $X \setminus y$ over $X \setminus \{x,y\}$. We identify $X'$ with $X$ as a set. By definition, the orders $<'$ and $\prec'$ agree with $<$ and $\prec$ on all pairs of elements, except possibly $(x,y)$ and $(y,x)$; thus to determine amalgamations, we simply have to determine if and how one can modify the orders on these particular pairs. If $x$ and $y$ are not $<$-adjacent, then $<$ cannot be modified: for example, if there is some $z$ such that $x<z<y$ then we must have $x <' z$ and $z <' y$ and so $x <' y$ since $<'$ is transitive. If $x$ and $y$ are $<$-adjacent, then $<$ can be modified, that is, we can flip the relative order of $x$ and $y$. A similar analysis holds for $\prec$. Thus, we can summarize the situation as follows:
\begin{itemize}
\item If $x$ and $y$ are not neighbors then $X$ itself is the unique proper amalgamation.
\item If $x$ and $y$ are $<$-adjacent but not $\prec$-adjacent then there are two proper amalgamations: $X$, and the amalgamation obtained by inverting the $<$-order of $x$ and $y$.
\item If $x$ and $y$ are $\prec$-adjacent but not $<$-adjacent, then there are two proper amalgamations, similar to the previous case.
\item If $x$ and $y$ are both $<$-adjacent and $\prec$-adjacent, then there are four proper amalgamations: one can invert $x$ and $y$ in the $<$ order, the $\prec$ order, both, or neither.
\end{itemize}
In particular, we see that $x$ and $y$ are separated if and only if they are not neighbors. It follows from this description that ``separated'' and ``stably separated'' coincide for permutations. Finally, we note that there is an improper amalgamation (i.e., one in which $x$ and $y$ are identified) if and only if $x$ and $y$ are both $<$-adjacent and $\prec$-adjacent.

\subsection{Minimal marked permutations}

An element in a permutation has at most four neighbors. Thus, by the analysis above, a minimal marked permutation has at most five elements. We wrote a simple computer program to output all minimal marked permutations. There are 87 of them, listed in Figure~\ref{f:minimal}. We thus see that $P^*$ is the polynomial ring in variables $x_1, \ldots, x_{87}$.

We claim that there are no equivalences among minimal marked permutations (other than isomorphisms). To see this, define the \defn{neighborhood} of an element $x$ in a permutation $X$, denoted $N_x(X)$, to be the set consisting of $x$ and its neighbors, endowed with the induced relational structure; we regard it as marked by taking $x$ to be the marked point. Clearly, $N_x(X)$ is a minimal marked permutation. If $(X', x)$ is a modification of $(X,x)$, meaning $X'$ is obtained by adding or deleting an element that is separated from $x$, then $N_x(X)=N_x(X')$. This shows that $N_x(X)$ is the unique minimal marked permutation that is equivalent to $(X,x)$, from which the claim follows. As a corollary, we find $\fb_0^*=0$.

\begin{figure} \small\centering
\begin{tabular}{cc|cc|cc|cc|cc}
$x_1$  & \m{1}      & $x_{19}$ & 3\m{2}1  & $x_{37}$ & \m{3}124  & $x_{55}$ & 1\m{3}542 & $x_{73}$ & 352\m{4}1 \\
$x_2$  & \m{1}2     & $x_{20}$ & 1\m{2}43 & $x_{38}$ & 3\m{1}42  & $x_{56}$ & 1\m{4}235 & $x_{74}$ & 41\m{3}52 \\
$x_3$  & 1\m{2}     & $x_{21}$ & 13\m{2}4 & $x_{39}$ & 314\m{2}  & $x_{57}$ & 14\m{2}53 & $x_{75}$ & 421\m{3}5 \\
$x_4$  & 2\m{1}     & $x_{22}$ & 1\m{3}24 & $x_{40}$ & \m{3}142  & $x_{58}$ & 1\m{4}253 & $x_{76}$ & 4\m{2}513 \\
$x_5$  & \m{2}1     & $x_{23}$ & 134\m{2} & $x_{41}$ & 31\m{4}2  & $x_{59}$ & 15\m{2}43 & $x_{77}$ & 4\m{2}531 \\
$x_6$  & 1\m{2}3    & $x_{24}$ & 1\m{3}42 & $x_{42}$ & 3\m{2}41  & $x_{60}$ & 21\m{3}54 & $x_{78}$ & 425\m{3}1 \\
$x_7$  & \m{1}32    & $x_{25}$ & 14\m{2}3 & $x_{43}$ & 32\m{4}1  & $x_{61}$ & 241\m{3}5 & $x_{79}$ & 45\m{3}12 \\
$x_8$  & 13\m{2}    & $x_{26}$ & 1\m{4}23 & $x_{44}$ & 34\m{2}1  & $x_{62}$ & 2\m{4}135 & $x_{80}$ & 51\m{4}23 \\
$x_9$  & 1\m{3}2    & $x_{27}$ & 21\m{3}4 & $x_{45}$ & 4\m{1}32  & $x_{63}$ & 2\m{4}153 & $x_{81}$ & 5\m{2}413 \\
$x_{10}$ & 2\m{1}3  & $x_{28}$ & 23\m{1}4 & $x_{46}$ & 41\m{3}2  & $x_{64}$ & 245\m{3}1 & $x_{82}$ & 52\m{4}13 \\
$x_{11}$ & \m{2}13  & $x_{29}$ & 2\m{3}14 & $x_{47}$ & 4\m{2}13  & $x_{65}$ & 25\m{3}14 & $x_{83}$ & 5\m{2}431 \\
$x_{12}$ & 21\m{3}  & $x_{30}$ & 24\m{1}3 & $x_{48}$ & 421\m{3}  & $x_{66}$ & 314\m{2}5 & $x_{84}$ & 5\m{3}124 \\
$x_{13}$ & 23\m{1}  & $x_{31}$ & \m{2}413 & $x_{49}$ & 4\m{2}31  & $x_{67}$ & 31\m{4}25 & $x_{85}$ & 5\m{3}142 \\
$x_{14}$ & \m{2}31  & $x_{32}$ & 241\m{3} & $x_{50}$ & 42\m{3}1  & $x_{68}$ & 315\m{2}4 & $x_{86}$ & 531\m{4}2 \\
$x_{15}$ & 2\m{3}1  & $x_{33}$ & 2\m{4}13 & $x_{51}$ & 4\m{3}12  & $x_{69}$ & 32\m{4}15 & $x_{87}$ & 532\m{4}1 \\
$x_{16}$ & 3\m{1}2  & $x_{34}$ & \m{2}431 & $x_{52}$ & 134\m{2}5 & $x_{70}$ & 34\m{2}51 \\
$x_{17}$ & 31\m{2}  & $x_{35}$ & 24\m{3}1 & $x_{53}$ & 135\m{2}4 & $x_{71}$ & 351\m{4}2 \\  
$x_{18}$ & \m{3}12  & $x_{36}$ & 31\m{2}4 & $x_{54}$ & 1\m{3}524 & $x_{72}$ & 35\m{2}41
\end{tabular}
\caption{The 87 minimal marked permutations. The marked point is indicated with a dot.}
\label{f:minimal}
\end{figure}

\subsection{Examples of equations}

Before proceeding to the general analaysis of the ideals $\fb_1^*$ and $\fb_2^*$, we give some examples illustrating how to work with L-data and Q-data. The equations in Appendices~\ref{s:linear} and~\ref{s:quadratic} are obtained just like these examples. We prefer to write the relations associated to L- and Q-data as equations, as opposed to ideal generators.

Consider an L-datum $(X,a,b)$ where $X$ is a permutation and $a$ and $b$ are distinct points. Let $X_1, \ldots, X_n$ be the proper amalgamations of $X \setminus a$ and $X \setminus b$ over $X \setminus \{a,b\}$, and let $\delta$ be~1 if there is an improper amalgamation and~0 otherwise. Then the equation associated to this L-datum is
\begin{displaymath}
\lbb X \setminus b, a \rbb = \delta + \sum_{i=1}^n \lbb X_i, a \rbb.
\end{displaymath}
To be more precise, the difference of the two sides is the associated generator of $\fb_1$; to get the associated generator of $\fb_1^*$, replace each marked permutation with its corresponding minimal marked permutation. Note that $a$ remains the marked point in all terms.

According to the analysis in \S \ref{ss:amalg}, there are four cases that determine how amalgamations behave. Thus there are four types of linear equations associated to L-data. We give examples of each type below. To actually write down an L-datum, we use a solid dot to indicate the first marked point $a$, and an open dot to indicate the second marked point $b$, e.g., \m{1}2\mm{3}.

Let us first consider the L-datum just mentioned, namely, \m{1}2\mm{3}. The left side of the equation is the class of the marked structure obtained by deleting the second marked point. In this case, the two marked points are not neighbors, as 2 is between 1 and 3 in both orders. Thus there is no proper amalgamation besides the one we started with, and there is no improper amalgamation. The linear equation associated to this L-datum is thus
\begin{displaymath}
[\dot{1}2] = [\dot{1}23].
\end{displaymath}
This is simply the identity which equates the class of \m{1}23 with the class of its corresponding minimal marked permutation \m{1}2. This case does not contribute a generator to $\fb_1^*$.

Next consider the L-datum 1\mm{2}45\m{3}. Deleting the second marked point yields a marked permutation isomorphic to 134\m{2}, which corresponds to $x_{23}$ in Figure~\ref{f:minimal}. The two marked points are $<$-adjacent but not $\prec$-adjacent. Thus there are two proper amalgamations and no improper amalgamation ($\delta=0$). The first amalgamation is the one we are given, with the first marked point remaining marked, that is, 1245\m{3}. This is not minimal, since~1 is separated from~3. Deleting~1, we obtain a marked permutation isomorphic to 134\m{2}, which is again $x_{23}$. The second amalgamation is obtained by inverting the $<$ relation of 2 and 3. This is the permutation with underlying set $\{1,\ldots,5\}$ and orders
\begin{displaymath}
1<3<2<4<5, \qquad 1 \prec 2 \prec 4 \prec 5 \prec 3.
\end{displaymath}
This isomorphic to the marked permutation 1345\m{2}; note that we have essentially renamed~2 and~3, which is why~2 is now marked. This is not minimal, since 4 is not a neighbor to~2. Deleting~4, we obtain a marked permutation isomorphic to 134\m{2}, which is once again $x_{23}$. Thus the equation associated to this L-datum is
\begin{displaymath}
x_{23}=x_{23}+x_{23},
\end{displaymath}
which shows that $x_{23}$ vanishes in $\Theta(\fP)$. (This equation is actually the generator of $\fb_1^*$ associated to the L-datum.) Note that this equation is the first one in Group~I in Appendix~\ref{s:linear}.

Next consider the L-datum 13\m{2}\mm{4}. The left side of the equation is 13\m{2}, which is $x_8$. The two marked points are $\prec$-adjacent, but not $<$-adjacent. The two amalgamations are 13\m{2}4 and 134\m{2}. These are both minimal, and correspond to $x_{21}$ and $x_{23}$. The associated equation is thus
\begin{displaymath}
x_8=x_{21}+x_{23}.
\end{displaymath}
Combined with the previous paragraph, which showed $x_{23}=0$, we obtain $x_8=x_{21}$. This is the first equation in Group~II in Appendix~\ref{s:linear}.

Finally, consider the L-datum \m{1}\mm{2}. The left side of the equation is \m{1}, which is $x_1$. The two marked points are adjacent in both orders. The four amalgamations are (isomorphic to)
\begin{displaymath}
\dot{1}2, \quad 1\dot{2}, \quad 2\dot{1}, \quad 2\dot{1}.
\end{displaymath}
Note that when we switch~1 and~2 in the first order, we are essentially renaming~1 and~2, which is why the marked point changes; we already saw this phenomenon when analyzing the L-datum 1\mm{2}45\m{3} above. The four marked permutations above correspond to $x_2$, $x_3$, $x_4$, and $x_5$. We have $\delta=1$ in this case, and so the resulting equation is
\begin{displaymath}
x_1=1+x_2+x_3+x_4+x_5.
\end{displaymath}
This is the first equation in Group~IV in Appendix~\ref{s:linear}.

We now move on to Q-data. Since the two marked points in a Q-datum play the same role, we use a solid dot to indicate both. Consider the Q-datum \m{1}\m{3}24. The quadratic equation associated to it is
\begin{displaymath}
[\dot{1}324] \cdot [\dot{3}24] = [1\dot{3}24] \cdot [\dot{1}24].
\end{displaymath}
Replacing each factor with its associated minimal marked class, we obtain
\begin{displaymath}
x_7 x_{11} = x_{22} x_2.
\end{displaymath}
This is the generator of $\fb_2^*$ corresponding to our original Q-datum. Applying to this the Group~II linear equations $x_9=x_{11}=x_{22}$, we obtain the equivalent equation
\begin{displaymath}
x_7 x_9 = x_2 x_9,
\end{displaymath}
which is the first equation in the middle column of Appendix~\ref{s:quadratic}. This last step is explained in more detail in \S \ref{ss:quadratic} below.

\subsection{Linear equations}

If $(X,a,b)$ is a minimal L-datum then $(X \setminus b, a)$ is a minimal marked permutation, and so $X$ has at most six vertices. We wrote a computer program that listed all linear equations coming from L-data with at most six vertices. There were 547 distinct equations on this list. By examining this list carefully, we found~79 particularly nice equations that generated all others, and double-checked this with a MatLab computation. These equations (and the L-data giving rise to them) are listed in Appendix~\ref{s:linear}. These~79 equations generate the ideal $\fb_1^*$ of linear equations.

We make a few comments on the structure of these equations. They are collected into four groups. Group~I has 48 equations of the form $x_i=0$. The 48 values of $i$ are:
\begin{align*}
& 23, 26, 28, 30, 31, 32, 33, 34, 37, 38, 39, 40, 41, 43, 45, 48, \\
& 52, 53, 54, 55, 56, 57, 58, 59, 61, 62, 63, 64, 66, 67, 68, 69, \\
& 70, 71, 72, 73, 75, 76, 77, 78, 80, 81, 82, 83, 84, 85, 86, 87.
\end{align*}
Group~II has~19 equations of the form $x_i=\pm x_j$. These equations are summarized as follows:
\begin{align*}
x_8    &=  x_{10} =  x_{21} = x_{35} = x_{42} \\
x_9    &=  x_{11} =  x_{22} = x_{46} = x_{47} \\
x_{14} &=  x_{16} =  x_{24} = x_{25} = x_{49} \\
x_{15} &=  x_{17} =  x_{29} = x_{36} = x_{50} \\
x_{60} &= -x_{65} = -x_{74} = x_{79}
\end{align*}
By using just the equations in Group~I and~II, we see that each $x_i$ is equal to one of the following 21~variables modulo $\fb_1^*$:
\begin{align} \label{eq:21vars}
\begin{split}
& x_1, x_2, x_3, x_4, x_5, x_8, x_9, x_{14}, x_{15}, x_{60}, x_{65} \\
& x_6, x_7, x_{12}, x_{13}, x_{18}, x_{19}, x_{20}, x_{27}, x_{44}, x_{51}.
\end{split}
\end{align}
Additionally, $x_{65}=-x_{60}$ modulo $\fb_1^*$. The 10 equations in Group~III express the~10 variables on the second line above in terms of the ones on the first line; these equations are homogeneous, i.e., there are no constant terms. Finally, Group~IV contains two inhomogeneous equations among the variables on the first line. To conclude, we see that modulo $\fb_1^*$, each variable can be expressed as a homogeneous linear combination of the~10 variables
\begin{equation} \label{eq:10vars}
x_1, x_2, x_3, x_4, x_5, x_8, x_9, x_{14}, x_{15}, x_{60},
\end{equation}
and these variables satisfy two inhomogeneous linear equations. Thus $P^*/\fb_1^*$ is isomorphic to a polynomial ring in eight variables; in fact, one can use the above list of variables with $x_5$ and $x_{60}$ removed. Note that the above~10 variables index the columns in Appendix~\ref{s:meas}.

\subsection{Quadratic equations} \label{ss:quadratic}

Let $(X,a,b)$ be a Q-datum. Suppose there is some $c \ne a,b$ in $X$ that does not neighbor $a$ or $b$. Then $c$ does not neighbor $b$ in $X \setminus a$. Indeed, say $c$ were $<$-adjacent to $b$ in $X \setminus a$. Then it must be the case that $a$ is the unique element of $X$ between $c$ and $b$ under $<$. But this means $c$ is $<$-adjacent to $a$ in $X$, a contradiction. We conclude that the existence of such a $c$ implies $(X,a,b)$ is not minimal.

Suppose now that $(X,a,b)$ is a minimal Q-datum. We claim that $X$ has at most eight vertices. Indeed, $a$ and $b$ must be neighbors by minimality. By the previous paragraph, every other vertex of $X$ must neighbor either $a$ or $b$. Now, $a$ has at most four neighbors, and so at most three other than $b$; similarly, $b$ has at most three other than $a$. We thus see that there can be at most six elements of $X \setminus \{a,b\}$, which proves the claim.

We wrote a computer program to output all quadratic equations coming from Q-data with at most eight vertices. This yielded a list of 1,404 equations. Since this is a large number of equations, we introduced a bit of optimization. After computing each equation, we converted each variable to one of the~21 variables in \eqref{eq:21vars} using the linear equations from Group~I and~II. This produced the list of~90 equations in Appendix~\ref{s:quadratic}. These equations generate the image of $\fb_2^*$ in $P^*/\fb_1^*$.

\subsection{The main theorem}

The ring $\Theta(\fP)$ is isomorphic to $P^*/\fb^*$, where $P^*$ is the polynomial ring in the 87~variables $x_1, \ldots, x_{87}$, and $\fb^*$ is the ideal generated by the 79~linear equations in Appendix~\ref{s:linear} and the 90~quadratic equations in Appendix~\ref{s:quadratic}. We used the computer algebra system Macaulay2 to analyze this ring, as follows.

We first created the polynomial ring $R$ over the rational numbers in the~10 variables listed in \eqref{eq:10vars}. We then defined the ideal $I$ generated by the two inhomogeneous linear equations in Group~IV, together with the~90 quadratic equations in Appendix~\ref{s:quadratic}, after replacing the variables there with the appropriate homogeneous linear combination of the~10 variables we are now using. The \texttt{basis} function tells us that $R/I$ has dimension~37 as a vector space. The \texttt{rationalPoints} function gives~37 distinct rational solutions to the equations. This proves that $\Theta(\fP) \otimes \bQ$ is isomorphic, as a ring, to $\bQ^{37}$. We note that the solutions to these equations are exactly giving the values of the~37 measures on the~10 classes in \eqref{eq:10vars}, as recorded in Appendix~\ref{s:meas}.

To complete the proof integrally, we continued as follows. We recreated $R$ and $I$ over $\bZ$. Calling the \texttt{basis} function shows that $R/I$ is generated as a $\bZ$-module by at most~47 elements. (Over non-fields, this function only gives a generating set, despite its name.) Calling the \texttt{leadingTerm} function, we see that the initial ideal of $I$ is generated by monomials whose coefficients are either~1 or~2; this shows that $(R/I)[1/2]$ is flat over $\bZ[1/2]$. Working modulo~2 and calling the \texttt{basis} function again, we find that $(R/I) \otimes \bF_2$ has dimension~37 as an $\bF_2$-vector space. All of this implies that $\Theta(\fP)=R/I$ is a free $\bZ$-module of rank~37. Since its~37 rational points all have integer coordinates and remain distinct modulo any prime, it follows that $\Theta(\fP)$ is isomorphic to $\bZ^{37}$ as a ring.

\section{The support of a measure} \label{s:supp}

Let $\mu$ be a measure on a Fra\"iss\'e class $\fF$ valued in a domain $k$. For a structure $X$, define $\mu(X)=\mu(\emptyset \subset X)$. We say that $X$ is \defn{regular} if $\mu(X)$ is a unit of $k$ for all $X$. Define the \defn{support} of $\mu$ to be the class $\fF'$ of structures $X$ for which $\mu(X) \ne 0$. Then $\fF'$ is a Fra\"iss\'e subclass of $\fF$, and $\mu$ induces a measure $\mu'$ on $\fF'$; see \cite[\S 2.7]{arboreal}. Note that $\mu'$ is regular if $k$ is a field, or if we regard it as valued in $\operatorname{Frac}(k)$.

We now consider the above concepts for our measures on the class $\fP$ of permutations. Cameron \cite{Cameron} showed that there are exactly seven Fra\"iss\'e subclasses of $\fP$:
\begin{itemize}
\item The class $\fP_1$ consisting of just the empty permutation.
\item The class $\fP_2$ consisting of permutations of cardinality at most one.
\item The class $\fP_3$ of permutations where the two orders coincide. It is defined by the excluded pattern 21, meaning $\fP_3$ consists of permutations which do not embed 21.
\item The class $\fP_4$ of permutations where the two orders are opposite. It is defined by the excluded pattern 12.
\item The class $\fP_5$ consisting of increasing sequences of decreasing sequences, defined by the excluded patterns 231 and 312.
\item The class $\fP_6$ consisting of decreasing sequences of increasining sequences, defined by the excluded patterns 213 and 132.
\item The class $\fP_7=\fP$ of all permutations.
\end{itemize}
If $\mu$ is any measure for $\fP$ then the support of $\mu$ is one of the above seven classes. The following table shows the support of each measure (see Appendix~\ref{s:meas} for the measures):
\vskip 3pt
\begin{center} \small
\begin{tabular}{c|cccccc}
Class & $\fP_1$ & $\fP_2$ & $\fP_3$ & $\fP_4$ & $\fP_5$ & $\fP_6$ \\
Measure(s) & 1--20 & 21--27 & 28--31 & 32--35 & 36 & 37
\end{tabular}
\end{center}
\vskip 2pt
There is no measure with support $\fP_7$, meaning there is no regular measure on $\fP$. Indeed, since $x_{23}=0$ in $\Theta(\fP)$, every measure vanishes on the permutation 1342.

\section{Tensor categories} \label{s:tensor}

Our interest in measures stems from applications to tensor categories. We close with a few brief remarks about how measures on permutations relate to tensor categories.

\subsection{Background}

Let $\fF$ be a Fra\"iss\'e class, and let $\mu$ be a measure for $\fF$ valued in a field $k$. In \cite[\S 4.3]{arboreal}, we define $k$-linear tensor category $\uPerm_k(\fF; \mu)$, which is essentially equivalent to the construction in \cite[\S 8]{repst}. The objects of this category are formal direct sums of symbols $\Vec_X$, where $X$ belongs to $\fF$, and morphisms $\Vec_X \to \Vec_Y$ are formal linear combinations of amalgamations of $X$ and $Y$. The structure constants for composition involve the measure.

The general expectation is that if $\mu$ is regular then one can hope that the Karoubi envelope $\uPerm_k(\fF; \mu)^{\rm kar}$ of $\uPerm_k(\fF; \mu)$ is semi-simple; this is not always true though. The following result gives a positive result in this direction. See \cite[Proposition~4.1]{arboreal} for a proof, and \cite[\S 4.2]{arboreal} for the definition of pre-Tannakian category.

\begin{proposition} \label{prop:ss1}
Suppose that $\mu$ is regular. Additionally, suppose the following:
\begin{enumerate}
\item $\mu$ lifts to a $\bZ[1/N]$-valued measure on $\fF$ for some $N$.
\item Only finitely many prime numbers occur as as divisors of $\# \Aut(X)$, for $X \in \fF$.
\item Given $Y \in \fF$, a proper substructure $X$, and an integer $h \ge 1$, there exists an embedding $X \to Z$ in $\fF$ that admits at least $h$ distinct extensions to $Y$.
\end{enumerate}
Then $\uPerm_k(\fF; \mu)^{\rm kar}$ is a semi-simple pre-Tannakian category.
\end{proposition}

Let $\fF'$ be the support of $\mu$, and let $\mu'$ be the induced regular measure on $\fF'$. The following result is \cite[Corollary~4.6]{arboreal}. See \cite[\S 2.3]{EtingofOstrik} for background on semi-simplification.

\begin{proposition} \label{prop:ss2}
Suppose $\uPerm_k(\fF'; \mu')^{\rm kar}$ is semi-simple. Then nilpotent endomorphisms in $\uPerm_k(\fF; \mu)^{\rm kar}$ have categorical trace~0, and the semi-simpliciation of $\uPerm_k(\fF; \mu)^{\rm kar}$ is $\uPerm_k(\fF'; \mu')^{\rm kar}$.
\end{proposition}

\subsection{Results}

Let $\mu$ be one of the~37 measures for $\fP$. Let $\fP'$ be the support of $\mu$, and let $\mu'$ be the restriction of $\mu$ to $\fP'$. Let $k$ be a field, and regard $\mu$ and $\mu'$ as $k$-valued.

\begin{theorem}
We have the following:
\begin{enumerate}
\item The category $\uPerm_k(\fP'; \mu')^{\rm kar}$ is a semi-simple pre-Tannakian category.
\item Nilpotent endomorphisms in $\uPerm_k(\fP; \mu)$ have categorical trace~0.
\item The semisimplification of $\uPerm_k(\fP; \mu)^{\rm kar}$ is $\uPerm_k(\fP'; \mu')^{\rm kar}$.
\end{enumerate}
\end{theorem}

\begin{proof}
If $\fP'$ is $\fP_1$ or $\fP_2$ of $\fP$ then $\uPerm_k(\fP'; \mu')$ is just the category of finite dimensional vector spaces, and so (a) holds. Suppose now that $\fP'$ is non-trivial. Proposition~\ref{prop:ss1}(a) holds with $N=1$. Proposition~\ref{prop:ss1}(b) holds since any automorphism of a permutation is trivial. It is not difficult to verify that $\fP'$ satisfies Proposition~\ref{prop:ss1}(c). Thus Proposition~\ref{prop:ss1} gives statement~(a). Statements (b) and (c) now follow from Proposition~\ref{prop:ss2}.
\end{proof}

If $\fP'$ is $\fP_3$ or $\fP_4$ then the pre-Tannakian category in (a) is equivalent to the Delannoy category studied in \cite{line}. If $\fP'$ is $\fP_5$ or $\fP_6$ then the pre-Tannakian category in (a) is new. The automorphism group of the Fra\"iss\'e limit of $\fP_5$ or $\fP_6$ is the wreath product $G \wr G$ appearing in Corollary~\ref{cor:wreath}, and so the above theorem in these cases implies the corollary.

\clearpage
\newpage
~\vfill
\appendix
\section{Linear equations} \label{s:linear}

\noindent
\textit{Group I (48 equations).} These equations all have the form $x_i=0$. In the following table, we simply write the variable instead of the whole equation. Each L-datum here actually gives the equation $x_i=x_i+x_i$, which simplifies to $x_i=0$.
\vskip 2ex
\begin{center}
\begin{tabular}{cc|cc|cc|cc|cc|cc}
$x_{23}$ &  1\mm{2}45\m{3} & $x_{37}$ &  \m{3}12\mm{4}5 & $x_{52}$ & 1\mm{2}45\m{3}6 &
 $x_{61}$ & 2\mm{3}51\m{4}6 & $x_{70}$ & 34\m{2}\mm{5}61 & $x_{80}$ & 61\mm{2}\m{5}34 \\
$x_{26}$ &  1\mm{2}\m{5}34 & $x_{38}$ &  3\m{1}\mm{4}52 & $x_{53}$ & 1\mm{2}46\m{3}5 &
 $x_{62}$ & 2\m{4}13\mm{5}6 & $x_{71}$ & 3\mm{4}61\m{5}2 & $x_{81}$ & 6\m{2}51\mm{3}4 \\
$x_{28}$ &  23\m{1}\mm{4}5 & $x_{39}$ &  4\mm{2}15\m{3} & $x_{54}$ & 1\m{3}62\mm{4}5 &
 $x_{63}$ & 2\m{4}16\mm{5}3 & $x_{72}$ & 36\m{2}\mm{5}41 & $x_{82}$ & 63\mm{2}\m{5}14 \\
$x_{30}$ &  25\m{1}\mm{4}3 & $x_{40}$ &  \m{3}15\mm{4}2 & $x_{55}$ & 1\m{3}65\mm{4}2 &
 $x_{64}$ & 2\mm{3}56\m{4}1 & $x_{73}$ & 3\mm{4}62\m{5}1 & $x_{83}$ & 6\m{2}54\mm{3}1 \\
$x_{31}$ &  \m{2}51\mm{3}4 & $x_{41}$ &  41\mm{2}\m{5}3 & $x_{56}$ & 1\mm{2}\m{5}346 &
 $x_{66}$ & 314\m{2}\mm{5}6 & $x_{75}$ & 5\mm{3}21\m{4}6 & $x_{84}$ & 6\m{3}12\mm{4}5 \\
$x_{32}$ &  2\mm{3}51\m{4} & $x_{43}$ &  43\mm{2}\m{5}1 & $x_{57}$ & 14\m{2}\mm{5}63 &
 $x_{67}$ & 41\mm{2}\m{5}36 & $x_{76}$ & 4\m{2}\mm{5}613 & $x_{85}$ & 6\m{3}15\mm{4}2 \\
$x_{33}$ &  3\mm{2}\m{5}14 & $x_{45}$ &  5\m{1}\mm{4}32 & $x_{58}$ & 1\mm{2}\m{5}364 &
 $x_{68}$ & 316\m{2}\mm{5}4 & $x_{77}$ & 4\m{2}\mm{5}631 & $x_{86}$ & 641\mm{2}\m{5}3 \\
$x_{34}$ &  \m{2}54\mm{3}1 & $x_{48}$ &  5\mm{3}21\m{4} & $x_{59}$ & 16\m{2}\mm{5}43 &
 $x_{69}$ & 43\mm{2}\m{5}16 & $x_{78}$ & 5\mm{3}26\m{4}1 & $x_{87}$ & 643\mm{2}\m{5}1
\end{tabular}
\end{center}
\vskip 2ex

\noindent
\textit{Group II (19 equations).}
In the following, one must combine the equation coming from the given L-datum with a Group~I equation to obtain the stated equation. For example, the first L-datum 13\m{2}\mm{4} gives $x_8=x_{21}+x_{23}$, and we have $x_{23}=0$.
\vskip 2ex
\begin{center}
\begin{tabular}{Ac|Ac|Ac|Ac}
x_{21}&=x_8    & 13\m{2}\mm{4}  & x_{21}&=x_{10} & \mm{1}3\m{2}4 &
x_{21}&=x_{35} & 24\mm{1}\m{3}5 & x_{21}&=x_{42} & \mm{2}4\m{3}51 \\
x_{22}&=x_9    & 1\m{3}2\mm{4}  & x_{22}&=x_{11} & \mm{1}\m{3}24 &
x_{22}&=x_{46} & \mm{3}1\m{4}25 & x_{22}&=x_{47} & 5\mm{1}\m{3}24 \\
x_{49}&=x_{14} & \m{2}\mm{4}31  & x_{49}&=x_{16} & 4\m{1}3\mm{2} &
x_{49}&=x_{24} & 1\m{3}\mm{5}42 & x_{49}&=x_{25} & 1\m{3}\mm{5}42 \\
x_{50}&=x_{15} & \mm{3}2\m{4}1  & x_{50}&=x_{17} & 42\mm{1}\m{3} &
x_{50}&=x_{29} & \mm{3}2\m{4}15 & x_{50}&=x_{36} & 42\mm{1}\m{3}5
\end{tabular}
\vskip 3pt
\begin{tabular}{cc|cc|cc}
$x_{60}=-x_{65}$ & 25\mm{1}\m{3}64 & $x_{60}=-x_{74}$ & 251\m{3}6\mm{4} &
$x_{79}=-x_{74}$ & 51\m{4}\mm{6}23
\end{tabular}
\end{center}
\vskip 2ex

\noindent
\textit{Group III  (10 equations).} In the first four rows, one must combine the equation coming from the L-datum with equations from Group~I and~II to obtain the stated equation. In the final row, one must also use equations from the preceeding rows. For example, 1\m{2}4\mm{3} gives $x_6=x_{20}+x_{24}$, and we have $x_{20}=x_{15}+x_{60}$ (first row below) and $x_{24}=x_{14}$ (Group~II).
\vskip 2ex
\begin{center}
\begin{tabular}{Ac|Ac}
x_{20} &= x_{15}+x_{60} & 2\mm{1}\m{3}54 & x_7    &= x_2-x_{14} & \m{1}3\mm{2} \\
x_{27} &= x_{14}+x_{60} & 21\m{3}5\mm{4} & x_{12} &= x_3-x_{15} & 2\mm{1}\m{3} \\
x_{44} &= x_9+x_{60}    & 45\mm{1}\m{3}2 & x_{13} &= x_4-x_8 & \mm{1}3\m{2} \\
x_{51} &= x_8+x_{60}    & \mm{3}5\m{4}12 & x_{18} &= x_5-x_9 & \mm{1}\m{3}2 \\
x_6    &= x_{14}+x_{15}+x_{60} & 1\m{2}4\mm{3} & x_{19} &= x_8+x_9+x_{60} & \mm{2}4\m{3}1
\end{tabular}
\end{center}
\vskip 2ex

\noindent
\textit{Group IV (2 equations).}
The first L-datum directly gives the stated equation. The second gives the equation $x_{17}=1+x_6+x_8+x_{17}+x_{46}$, which is equivalent to the stated equation using the equation for $x_6$ in Group~III, and the identity $x_{46}=x_9$ from Group~II.
\vskip 2ex
\begin{center}
\begin{tabular}{Ac}
0 &= 1-x_1+x_2+x_3+x_4+x_5 & \m{1}\mm{2} \\
0 &= 1+x_8+x_9+x_{14}+x_{15}+x_{60} & 41\m{2}\mm{3}
\end{tabular}
\end{center}
\vfill
\clearpage
\newpage

~\vfill
\section{Quadratic equations} \label{s:quadratic}
\vskip 1ex
The following table lists the 90 quadratic generators of the $\fb^*$ ideal, and the associated Q-data. We have applied linear equations in Group~I and~II to convert the variables appearing in these equations to the~21 listed in \eqref{eq:21vars}, as described in \S \ref{ss:quadratic}.
\vskip 3ex
\begin{center}
\begin{tabular}{Ac|Ac|Ac}
0&=x_7x_{13}         & \m{1}34\m{2} & x_2x_9&=x_7x_9             & \m{1}\m{3}24 & x_4x_{14}&=x_{13}x_{14}    & 4\m{2}3\m{1}   \\
0&=x_8x_{14}         & 1\m{3}4\m{2} & x_6x_8&=x_6x_9             & 1\m{3}\m{2}4 & x_4x_{15}&=x_{13}x_{15}    & 42\m{3}\m{1}   \\
0&=x_8x_{15}         & 13\m{4}\m{2} & x_3x_8&=x_8x_{12}          & 13\m{2}\m{4} & x_{14}x_{19}&=x_{15}x_{19} & 4\m{2}\m{3}1   \\
0&=x_7x_{18}         & \m{1}\m{4}23 & x_3x_9&=x_9x_{12}          & 1\m{3}2\m{4} & x_5x_{14}&=x_{14}x_{18}    & \m{4}\m{2}31   \\
0&=x_9x_{14}         & 1\m{4}\m{2}3 & x_9x_{14}&=x_9x_{15}       & 1\m{3}\m{4}2 & x_5x_{15}&=x_{15}x_{18}    & \m{4}2\m{3}1   \\
0&=x_9x_{15}         & 1\m{4}2\m{3} & x_8x_{14}&=x_8x_{15}       & 14\m{2}\m{3} & x_{14}x_{19}&=x_{14}x_{51} & 4\m{3}\m{1}2   \\
0&=x_{12}x_{13}      & 23\m{1}\m{4} & x_4x_7&=x_7x_8             & \m{1}43\m{2} & x_{15}x_{19}&=x_{15}x_{51} & 4\m{3}1\m{2}   \\
0&=x_{12}x_{18}      & \m{3}12\m{4} & x_5x_7&=x_7x_9             & \m{1}\m{4}32 & x_5x_{18}&=x_{18}x_{51}    & \m{4}\m{3}12   \\
x_1x_2&=x_1x_3       &   \m{1}\m{2} & x_8x_8&=x_8x_{19}          & 14\m{3}\m{2} & x_6x_{20}&=x_{20}x_{20}    & 1\m{2}\m{3}54  \\
x_1x_4&=x_1x_5       &   \m{2}\m{1} & x_9x_9&=x_9x_{19}          & 1\m{4}\m{3}2 & x_8x_{20}&=x_9x_{20}       & 1\m{3}\m{2}54  \\
x_2x_2&=x_2x_6       &  \m{1}\m{2}3 & x_6x_8&=x_8x_{27}          & 2\m{1}\m{3}4 & x_6x_{27}&=x_{27}x_{27}    & 21\m{3}\m{4}5  \\
x_3x_3&=x_3x_6       &  1\m{2}\m{3} & x_6x_9&=x_9x_{27}          & \m{2}1\m{3}4 & x_8x_{20}&=x_8x_{60}       & 2\m{1}\m{3}54  \\
x_2x_8&=x_4x_7       &  \m{1}3\m{2} & x_3x_{12}&=x_{12}x_{27}    & 21\m{3}\m{4} & x_9x_{20}&=x_9x_{60}       & \m{2}1\m{3}54  \\
x_2x_9&=x_5x_7       &  \m{1}\m{3}2 & x_7x_8&=x_7x_9             & \m{2}\m{1}43 & x_8x_{27}&=x_8x_{60}       & 21\m{3}5\m{4}  \\
x_3x_8&=x_3x_9       &  1\m{3}\m{2} & x_8x_{12}&=x_9x_{12}       & 21\m{4}\m{3} & x_9x_{27}&=x_9x_{60}       & 21\m{3}\m{5}4  \\
x_2x_8&=x_2x_9       &  \m{2}\m{1}3 & x_2x_{13}&=x_{13}x_{14}    & \m{2}34\m{1} & x_8x_{27}&=x_9x_{27}       & 21\m{4}\m{3}5  \\
x_3x_8&=x_4x_{12}    &  2\m{1}\m{3} & x_3x_{13}&=x_{13}x_{15}    & 23\m{4}\m{1} & x_{14}x_{44}&=x_{15}x_{44} & 45\m{2}\m{3}1  \\
x_3x_9&=x_5x_{12}    &  \m{2}1\m{3} & x_6x_{14}&=x_{14}x_{14}    & \m{2}\m{3}41 & x_{14}x_{44}&=x_{14}x_{60} & 45\m{3}\m{1}2  \\
x_2x_{13}&=x_4x_{14} &  \m{2}3\m{1} & x_6x_{15}&=x_{15}x_{15}    & 2\m{3}\m{4}1 & x_{15}x_{44}&=x_{15}x_{60} & 45\m{3}1\m{2}  \\
x_3x_{13}&=x_4x_{15} &  2\m{3}\m{1} & x_8x_{15}&=x_9x_{15}       & 2\m{4}\m{3}1 & x_{14}x_{51}&=x_{14}x_{60} & \m{4}5\m{3}12  \\
x_5x_{14}&=x_5x_{15} &  \m{2}\m{3}1 & x_4x_{12}&=x_8x_{12}       & 32\m{1}\m{4} & x_{15}x_{51}&=x_{15}x_{60} & 4\m{5}\m{3}12  \\
x_4x_{14}&=x_4x_{15} &  3\m{1}\m{2} & x_5x_{12}&=x_9x_{12}       & \m{3}21\m{4} & x_{19}x_{44}&=x_{44}x_{44} & 45\m{3}\m{2}1  \\
x_2x_{18}&=x_5x_{14} &  \m{3}\m{1}2 & x_8x_{14}&=x_9x_{14}       & \m{3}\m{2}41 & x_{14}x_{51}&=x_{15}x_{51} & 5\m{3}\m{4}12  \\
x_3x_{18}&=x_5x_{15} &  \m{3}1\m{2} & x_{13}x_{14}&=x_{13}x_{15} & 34\m{1}\m{2} & x_{19}x_{51}&=x_{51}x_{51} & 5\m{4}\m{3}12  \\
x_4x_4&=x_4x_{19}    &  3\m{2}\m{1} & x_{14}x_{18}&=x_{15}x_{18} & \m{3}\m{4}12 & x_{20}x_{60}&=x_{27}x_{60} & 21\m{3}\m{4}65 \\
x_5x_5&=x_5x_{19}    &  \m{3}\m{2}1 & x_4x_{13}&=x_{13}x_{44}    & 34\m{2}\m{1} & x_8x_{60}&=x_9x_{60}       & 21\m{4}\m{3}65 \\
x_2x_7&=x_7x_{20}    & \m{1}\m{2}43 & x_{14}x_{19}&=x_{14}x_{44} & \m{3}4\m{2}1 & x_{14}x_{65}&=x_{15}x_{65} & 26\m{3}\m{4}15 \\
x_6x_8&=x_8x_{20}    & 1\m{2}4\m{3} & x_{15}x_{19}&=x_{15}x_{44} & 3\m{4}\m{2}1 & x_8x_{65}&=x_9x_{65}       & 26\m{4}\m{3}15 \\
x_6x_9&=x_9x_{20}    & 1\m{2}\m{4}3 & x_2x_{18}&=x_{14}x_{18}    & \m{4}\m{1}23 & x_{14}x_{60}&=x_{15}x_{60} & 56\m{3}\m{4}12 \\
x_2x_8&=x_7x_8       & \m{1}3\m{2}4 & x_3x_{18}&=x_{15}x_{18}    & \m{4}12\m{3} & x_{44}x_{60}&=x_{51}x_{60} & 56\m{4}\m{3}12
\end{tabular}
\end{center}
\vfill
\clearpage
\newpage

~\vfill
\section{The thirty-seven measures} \label{s:meas}
\def\mp{+}
\def\mm{--}
\def\mz{$\cdot$}
The table gives the values of the 37 measures on the 10 classes in \eqref{eq:10vars}. To determine the value of a measure on a general $x_i$, express $x_i$ as a linear combination of these~10 classes using the relations in Appendix~\S \ref{s:linear}. We write \mz, \mp, \mm\ for 0, $+1$, and $-1$ below. The final column indicates the support of the measure (\S \ref{s:supp}).
\vskip 2ex
\begin{center}
\begin{tabular}{c|cccccccccc|c}
& $x_1$ & $x_2$ & $x_3$ & $x_4$ & $x_5$ & $x_8$ & $x_9$ & $x_{14}$ & $x_{15}$ & $x_{60}$ & Supp \\
\hline
1  &  \mz &  \mz &  \mz &  \mm &  \mz &  \mm &  \mz &  \mz &  \mz &  \mz
& \multirow{20}{*}{$\fP_1$} \\
2  &  \mz &  \mz &  \mz &  \mz &  \mm &  \mm &  \mz &  \mz &  \mz &  \mz \\
3  &  \mz &  \mz &  \mz &  \mm &  \mz &  \mz &  \mm &  \mz &  \mz &  \mz \\
4  &  \mz &  \mz &  \mz &  \mz &  \mm &  \mz &  \mm &  \mz &  \mz &  \mz \\
5  &  \mz &  \mm &  \mz &  \mz &  \mz &  \mz &  \mz &  \mm &  \mz &  \mz \\
6  &  \mz &  \mz &  \mm &  \mz &  \mz &  \mz &  \mz &  \mm &  \mz &  \mz \\
7  &  \mz &  \mm &  \mz &  \mz &  \mz &  \mz &  \mz &  \mz &  \mm &  \mz \\
8  &  \mz &  \mz &  \mm &  \mz &  \mz &  \mz &  \mz &  \mz &  \mm &  \mz \\
9  &  \mz &  \mm &  \mz &  \mz &  \mz &  \mz &  \mz &  \mz &  \mz &  \mm \\
10 &  \mz &  \mz &  \mm &  \mz &  \mz &  \mz &  \mz &  \mz &  \mz &  \mm \\
11 &  \mz &  \mz &  \mz &  \mm &  \mz &  \mz &  \mz &  \mz &  \mz &  \mm \\
12 &  \mz &  \mz &  \mz &  \mz &  \mm &  \mz &  \mz &  \mz &  \mz &  \mm \\
13 &  \mz &  \mz &  \mz &  \mm &  \mz &  \mm &  \mm &  \mz &  \mz &  \mp \\
14 &  \mz &  \mz &  \mz &  \mz &  \mm &  \mm &  \mm &  \mz &  \mz &  \mp \\
15 &  \mz &  \mm &  \mz &  \mz &  \mz &  \mz &  \mz &  \mm &  \mm &  \mp \\
16 &  \mz &  \mz &  \mm &  \mz &  \mz &  \mz &  \mz &  \mm &  \mm &  \mp \\
17 &  \mz &  \mp &  \mz &  \mm &  \mm &  \mm &  \mm &  \mz &  \mz &  \mp \\
18 &  \mz &  \mz &  \mp &  \mm &  \mm &  \mm &  \mm &  \mz &  \mz &  \mp \\
19 &  \mz &  \mm &  \mm &  \mp &  \mz &  \mz &  \mz &  \mm &  \mm &  \mp \\
20 &  \mz &  \mm &  \mm &  \mz &  \mp &  \mz &  \mz &  \mm &  \mm &  \mp \\
\hline
21 &  \mp &  \mz &  \mz &  \mz &  \mz &  \mm &  \mz &  \mz &  \mz &  \mz
& \multirow{7}{*}{$\fP_2$} \\
22 &  \mp &  \mz &  \mz &  \mz &  \mz &  \mz &  \mm &  \mz &  \mz &  \mz \\
23 &  \mp &  \mz &  \mz &  \mz &  \mz &  \mz &  \mz &  \mm &  \mz &  \mz \\
24 &  \mp &  \mz &  \mz &  \mz &  \mz &  \mz &  \mz &  \mz &  \mm &  \mz \\
25 &  \mp &  \mz &  \mz &  \mz &  \mz &  \mz &  \mz &  \mz &  \mz &  \mm \\
26 &  \mp &  \mz &  \mz &  \mz &  \mz &  \mm &  \mm &  \mz &  \mz &  \mp \\
27 &  \mp &  \mz &  \mz &  \mz &  \mz &  \mz &  \mz &  \mm &  \mm &  \mp \\
\hline
28 &  \mm &  \mm &  \mm &  \mz &  \mz &  \mz &  \mz &  \mm &  \mz &  \mz
& \multirow{4}{*}{$\fP_3$} \\
29 &  \mm &  \mm &  \mm &  \mz &  \mz &  \mz &  \mz &  \mz &  \mm &  \mz \\
30 &  \mm &  \mm &  \mm &  \mz &  \mz &  \mz &  \mz &  \mz &  \mz &  \mm \\
31 &  \mm &  \mm &  \mm &  \mz &  \mz &  \mz &  \mz &  \mm &  \mm &  \mp \\
\hline
32 &  \mm &  \mz &  \mz &  \mm &  \mm &  \mm &  \mz &  \mz &  \mz &  \mz
& \multirow{4}{*}{$\fP_4$} \\
33 &  \mm &  \mz &  \mz &  \mm &  \mm &  \mz &  \mm &  \mz &  \mz &  \mz \\
34 &  \mm &  \mz &  \mz &  \mm &  \mm &  \mz &  \mz &  \mz &  \mz &  \mm \\
35 &  \mm &  \mz &  \mz &  \mm &  \mm &  \mm &  \mm &  \mz &  \mz &  \mp \\
\hline
36 &  \mp &  \mp &  \mp &  \mm &  \mm &  \mm &  \mm &  \mz &  \mz &  \mp & $\fP_5$\\
37 &  \mp &  \mm &  \mm &  \mp &  \mp &  \mz &  \mz &  \mm &  \mm &  \mp & $\fP_6$
\end{tabular}
\end{center}
\vfill

\end{document}